\documentclass[11pt]{article}
\usepackage{graphicx}
\usepackage{url}
\usepackage{amsfonts,amsthm,amsmath}
\usepackage{amssymb}
\usepackage[top=1.25in, bottom=1.5in, left=1.5in, right=1.5in]{geometry}
\usepackage{xcolor}
\usepackage{array}
\newcolumntype{x}[1]{>{\centering\arraybackslash\hspace{0pt}}p{#1}}
\usepackage[font=normalsize,labelfont={bf}]{caption}

\usepackage[flushleft]{threeparttable}

\newtheorem{Theorem}{Theorem}[section]
\newtheorem{Definition}{Definition}[section]
\newtheorem{Example}{Example}[section]
\newtheorem{Remark}{Remark}[section]
\newtheorem{Lemma}[Theorem]{Lemma}

\newcommand{\TTT}{{\ensuremath{\mathcal{T}}}}

\newcommand{\LLL}{{\ensuremath{\mathcal{L}}}}

\newcommand{\zed}{{\ensuremath{\mathbb{Z}}}}

\newcommand{\G}{\ensuremath{\mathcal{G}}}

\newcommand{\A}{\ensuremath{\mathcal{A}}}
\newcommand{\Or}{\ensuremath{\mathcal{O}}}

\newcommand{\B}{\ensuremath{\mathcal{B}}}

\renewcommand*{\P}{\ensuremath{\mathcal{P}}}

\title{Weak and Strong Nestings of BIBDs}
\author{preliminary notes---do not distribute}
\author{Douglas R.\ Stinson\thanks{D.R.\ Stinson's research is supported by  NSERC discovery grant RGPIN-03882.}\\
David R.\ Cheriton School of Computer Science\\University of Waterloo\\ Waterloo ON, N2L 3G1, Canada}

\date{\today}

\begin{document}

\maketitle

\begin{abstract}
We study two types of nestings of balanced incomplete block designs (BIBDs). In both types of nesting, we wish to add a point (the \emph{nested point}) to every block of a 
$(v,k,\lambda)$-BIBD in such a way that we end up with a partial $(w,k+1,\lambda+1)$-BIBD for some $w \geq v$.
%This design with block size $k+1$ is called the \emph{augmented design}. 
In the case where $w > v$, we are introducing $w-v$ new points. This is called a \emph{weak nesting}. A \emph{strong nesting} satisfies the stronger property that no pair containing a new point occurs more than once in the partial $(w,k+1,\lambda+1)$-BIBD. In both cases, the goal is to minimize $w$. We prove lower bounds on $w$ as a function of $v$, $k$ and $\lambda$ and we find infinite classes of $(v,2,1)$- and $(v,3,2)$-BIBDs that have \emph{optimal nestings}. 
%The nestings with $v = w = r + (v+1)/2$ are termed \emph{perfect nestings} (where, as usual, $r = \lambda(v-1)/(k-1)$ is the \emph{replication number} of the $(v,k,\lambda)$-BIBD.)
%Such nestings were previously known to exist in the case of Steiner triple systems (i.e., $(v,3,1)$-BIBDs) when $v \equiv 1 \bmod 6$. %Here we study the next possible cases of 
%perfect nestings, namely $(v,5,2)$-BIBDs when $v \equiv 1 \bmod 10$.
\end{abstract}

\section{Introduction}

Various kinds of nested designs have been studied for many years. One definition of these objects can be found in 
\cite[VI.36]{CD}, where it is required that a balanced incomplete block design (BIBD) with blocks of size $dk$ can be decomposed into $d$ BIBDs with blocks of size $k$. Here we consider a problem motivated by nesting of Steiner triple systems (STS), e.g., as studied in \cite{St85}. 

A \emph{$(v,k,\lambda)$-balanced incomplete block design} (or \emph{BIBD}) is a pair $(X,\A)$ that satisfies the following properties:
\begin{enumerate}
\item $X$ is a set of $v$ \emph{points}.
\item $\A$ consists of a multiset of \emph{blocks} of size $k$ such that 
 every pair of points  is contained in exactly $\lambda$ blocks.
\end{enumerate}
It can be shown that every point in $X$ occurs in exactly $r = \lambda(v-1)/(k-1)$ blocks. Also, the total number of blocks is $b = vr/k = \lambda v(v-1) / (k(k-1))$.

A $(v,3,1)$-BIBD is called a \emph{Steiner triple system} and it is denoted by STS$(v)$. In an STS$(v)$, we have $r = (v-1)/2$ and $b = v(v-1)/6$.

An STS$(v)$, say $(X, \A)$, can be \emph{nested} if there is a mapping $\phi : \A \rightarrow X$ such that 
$(X, \{ A \cup \{\phi(A)\} : A \in \A )$ is a $(v,4,2)$-BIBD. In words, we are adding a fourth point to every block of the STS$(v)$ so the result is a $(v,4,2)$-BIBD. It was shown in \cite{St85} that, for all $v \equiv 1 \bmod 6$, there exists a nested STS$(v)$. 

More generally, we could start with an arbitrary $(v,k,\lambda)$-BIBD and again add a point to each block, hoping to obtain a \emph{partial} $(v,k+1,\lambda+1)$-BIBD. (In a partial $(v,k,\lambda)$-BIBD, every pair of points occurs in \emph{at most} $\lambda$ blocks.) The mapping $\phi : \A \rightarrow X$ is a \emph{nesting} provided that $(X, \{ A \cup \{\phi(A)\} : A \in \A \} )$ is a \emph{partial} $(v,k+1,\lambda+1)$-BIBD. %The point $\phi(A)$ is the \emph{nested point} for the block $A$. 
In order to distinguish these kinds of nestings from some generalizations that we will be discussing, we sometimes refer to them as \emph{minimal nestings}. 

If the partial $(v,k+1,\lambda+1)$-BIBD is in fact a $(v,k+1,\lambda+1)$-BIBD, we have a 
\emph{perfect nesting}. There have been various papers that have studied perfect nestings of BIBDs and group-divisible designs (GDDs), including \cite{POC,LR,St85} among others.

We recall a few relevant known results concerning minimal nestings.

\begin{Lemma} 
\textup{\cite{BKS}}
\label{bound.lem} Suppose that a $(v,k,\lambda)$-BIBD has a nesting. Then $k \geq 2 \lambda + 1$.
\end{Lemma}

\begin{Theorem}
\label{equiv.thm}
\textup{\cite{BKS}}
Suppose a $(v,k,\lambda)$-BIBD has a nesting. The following are equivalent:
\begin{enumerate}
\item $k = 2 \lambda + 1$, 
\item $v =  2r+1$, and
\item the nesting is perfect.
\end{enumerate}
\end{Theorem}

\begin{Theorem}
\textup{\cite{St85}}
There exists a nested $(v,3,1)$-BIBD if and only if $v \equiv 1  \bmod 6$ and  $v \geq 7$.
\end{Theorem}

\begin{Theorem}
\textup{\cite{BKS}}
There exists a nested $(v,4,1)$-BIBD if and only if $v \equiv 1 \text{ or } 4 \bmod 12$ and  $v \geq 13$.\end{Theorem}

We remark that the nestings of $(v,4,1)$-BIBDs are not perfect (this follows from Theorem \ref{equiv.thm}).

\begin{Lemma} 
\label{2k+1.lem}
\textup{\cite{BKS}}
If a $(v,k,\lambda)$-BIBD has a perfect nesting, then $v \equiv 1 \bmod 2k$.
\end{Lemma}

We recall some  terminology from \cite{BKS}.   The point $\phi(A)$ is called the \emph{nested point} for the block $A$. We refer to $A \cup \{ \phi(A) \}$ as an \emph{augmented block}. %The multiset $N_{\phi} = \{ \phi(A) : A \in \A\}$ consists of all the nested points, counting multiplicities. 
The (possibly partial) $(v,k+1,\lambda+1)$-BIBD obtained from a nesting is called the \emph{augmented design}. In the case where the nesting is perfect, we may also call it the \emph{augmented BIBD}.

If $k \leq 2 \lambda$, then a  $(v,k,\lambda)$-BIBD does not have a (minimal) nesting, from Lemma \ref{bound.lem}.
This suggests that we consider more general kinds of nestings. For example,  
we might allow ``new'' points (i.e., points that are not in the $(v,k,\lambda)$-BIBD) to be introduced.

\begin{Definition}
{\rm Suppose $(X, \A)$ is a $(v,k,\lambda)$-BIBD and let $\phi : \A \rightarrow Y$, where
$X \subseteq Y$. Denote $|Y| = w \geq v$. %We will consider two variants of these generalized nestings.

A mapping $\phi : \A \rightarrow Y$ is a \emph{weak nesting} provided that 
$(Y, \{ A \cup \{\phi(A)\} : A \in \A )$ is a \emph{partial} $(w,k+1,\lambda+1)$-BIBD.  

A mapping $\phi : \A \rightarrow Y$ is a \emph{strong nesting} if 
\begin{enumerate}
\item for all $A \in \A$, it holds that $\phi(A) \cap A = \emptyset$, and
\item the multiset
\[ \{ \{ x, \phi(A)\}: x \in A \in \A \} \] consists of \emph{distinct pairs} of points.
\end{enumerate}
}
\end{Definition}

In a weak nesting, any pair of points can occur up to $\lambda +1$ times. 
In a strong nesting, a pair that includes a new point (i.e., a point in $Y \setminus X$) can only occur once, whereas, 
pairs of old points (i.e., points in $X$) can occur up to $\lambda +1$ times in the augmented design.
%As in the case of a weak nesting, a nesting is \emph{optimal} if $w$ is as small as possible.

It is easy to see that a strong nesting is also a weak nesting. Also, any minimal nesting is automatically strong. 

%A variation of weak nesting, called a strong nesting, is introduced in Section \ref{strong.sec}.

%footnote{We will be considering various types of nestings in this paper. A quick summary of the terms we use can be found in Table \ref{nest.tab} in Section \ref{summary.sec}.}

If a minimal nesting of a $(v,k,\lambda)$-BIBD does not exist, 
a natural problem is to construct a strong or weak nesting of a $(v,k,\lambda)$-BIBD in which $w$ is as small as possible; such a nesting is an \emph{optimal} nesting.

We present two small examples to illustrate the difference between weak and strong nestings.

\begin{Example}
\label{E4}
{\rm We construct an optimal weakly nested $(4,3,2)$-BIBD.
The augmented design has blocks $(1,2,3,\infty_1)$, $(1,2,4,\infty_1)$, $(1,3,4,\infty_1)$, $(2,3,4,\infty_1)$.

The augmented design is a partial $(5,4,3)$-BIBD. The new point is $\infty_1$. 
It is easy to verify that there are no minimal nestings of the $(4,3,2)$-BIBD, so the 
nesting is optimal (the optimality of this nesting is also a consequence of Lemma \ref{bound.lem}). %, from (\ref{3weakbound-int}).
}
\end{Example}

\begin{Example}
\label{E4strong}
{\rm We construct an optimal strongly nested $(4,3,2)$-BIBD with $w=7$.
The augmented design has blocks $(1,2,3,4)$, $(1,2,4,\infty_1)$, $(1,3,4,\infty_2)$, $(2,3,4,\infty_3)$.
The augmented design has $w = 7$ points; the  new points are $\infty_1, \infty_2 , \infty_3$. 

To verify that the nesting is strong, we examine the new pairs occurring with the nested points:
$\{1,4\}$, $\{2,4\}$, $\{3,4\}$, $\{\infty_1 ,1\}$, $\{\infty_1 ,2\}$, $\{\infty_1 ,4\}$, $\{\infty_2 ,1\}$, 
$\{\infty_2 ,3\}$, $\{\infty_2 ,4\}$, $\{\infty_3 ,2\}$, $\{\infty_3 ,3\}$, $\{\infty_3 ,4\}$. These pairs are all distinct, so we have a strong nesting.

There are various ways to show that there is no strong nesting of this $(4,3,2)$-BIBD with $w\leq 6$, so the nesting with $w = 7$ is optimal. For example, the lower bound $w \geq 7$ follows from Lemma \ref{4mod6bound}, to be proven later in this paper. 
}
\end{Example}

For future reference, we summarize the differences between  various types of nestings in Table \ref{nest.tab}.
%These notions are discussed in detail in later sections.

\begin{table}
\caption{Summary of definitions of various types of nestings}
\label{nest.tab}
\begin{center}
\begin{tabular}{c|p{3.5in}}
type of nesting & \multicolumn{1}{c}{definition} \\ \hline
weak & augmented design is a partial BIBD with index $\lambda + 1$ \\ \hline
strong & augmented design is a partial BIBD with index $\lambda + 1$, and all the pairs in the augmented design that contain a nested point are distinct  \\ \hline
optimal & $w$ is as small as possible (applies to weak and strong nestings) \\ \hline
minimal & $w = v$ (this is necessarily a strong nesting) \\ \hline
perfect & a minimal nesting where the nested design is a BIBD with index $\lambda +1$, or equivalently, 
the equality  $k = 2 \lambda + 1$  is satisfied
\end{tabular}
\end{center}
\end{table}

\subsection{Design theory definitions}

Here we give a few definitions concerning the types of designs that we use in the rest of the paper.

A \emph{$(k,\lambda)$-group-divisible design} (or \emph{GDD}) %of type $t^u$ 
is a triple $(X,\G,\A)$ that satisfies the following properties:
\begin{enumerate}
\item $X$ is a set of \emph{points}.
\item $\G$ is a partition of $X$ into subsets called %$u$ 
\emph{groups}, %of size $t$, 
say $\G = \{G_1, \dots , G_u\}$.
\item $\A$ consists of a multiset of \emph{blocks} of size $k$ such that 
\begin{enumerate}
\item $|G_i \cap A| \leq 1$ for $1 \leq i \leq u$ and for all $A \in \A$,
\item every pair of points from different groups is contained in exactly $\lambda$ blocks.
\end{enumerate}
\end{enumerate}
If $\lambda = 1$, we often refer to the GDD as a $k$-GDD.

The \emph{type} of a GDD is the multiset $\{|G| : G \in \G\}$. If there are $u_i$ groups of size $t_i$, for 
$1 = i, \dots , s$, then we write the type as ${t_1}^{u_1} {t_2}^{u_2} \cdots {t_s}^{u_s}$. If all the groups have the same size, then the type is of the form $t^u$.

A $(k,\lambda)$-GDD of type $t^u$ can be \emph{nested} if there is a mapping 
$\phi : \A \rightarrow X$ such that
$(X,\G,\B)$  is a partial $(k+1,\lambda + 1)$-GDD of type $t^u$, where
\[ \B = \{ A \cup \{\phi(A)\} : A \in \A \} .\]
If the augmented partial GDD is actually a GDD (i.e., a $(k+1,\lambda + 1)$-GDD of type $t^u$), then we say that the nesting of the GDD is \emph{perfect}.

It has been shown in \cite{BKS} that a nesting of a $(k,\lambda)$-GDD is perfect if and only if $k = 2\lambda +1$. 
In this paper, we only make use of nestings of $(3,1)$-GDDs (i.e., $3$-GDDs). Since $3 = 2 \times 1 + 1$, these are perfect nestings.

A $(v,k,\lambda)$-BIBD, say $(X,\A)$, is \emph{resolvable} if the set of blocks $\A$ can be partitioned into 
$r = \lambda(v-1)/(k-1)$ subsets of blocks, say $\P_1, \dots , \P_{r}$, such that each $\P_i$ partitions the set of points $X$. The $\P_i$'s are called \emph{parallel classes}. Of  course $v \equiv 0 \bmod k$ is a necessary condition for a resolvable $(v,k,\lambda)$-BIBD to exist. A resolvable STS$(v)$ is a \emph{Kirkman triple system}; it is denoted by KTS$(v)$.
KTS$(v)$ are known to exist for all $v \equiv 3 \bmod 6$ (see \cite{CR}).

A $(k,\lambda)$-GDD of type $t^u$, say $(X,\G,\A)$,  is  \emph{resolvable} if the set of blocks $\A$ can be partitioned into $r = \lambda t(u-1)/(k-1)$ subsets of blocks, say $\P_1, \dots , \P_{r}$, such that each $\P_i$ partitions the set of points $X$. The $\P_i$'s are called \emph{parallel classes}. 
A resolvable $(k,\lambda)$-GDD of type $t^u$ can exist only if $tu \equiv 0 \bmod k$.

A \emph{3-frame} is a $3$-GDD, say $(X,\G,\A)$ such that $\A$ can be partitioned into 
\emph{holey parallel classes}, where each holey parallel class partitions $X \setminus G$ for some $G \in \G$.
If $\P \subseteq \A$ is a partition of $X \setminus G$, then we say that $G$ is the \emph{hole} associated with 
$\P$. It is known that, for every group (i.e., hole) $G \in \G$, there must be exactly $|G|/2$ holey parallel classes associated with $G$ (see \cite{St87}). The total number of holey parallel classes is $|X|/2$.

Finally, a $(v,k,\lambda)$-BIBD, say $(X,\A)$, is \emph{cyclic} if $X = \zed_v$ and $\A$ is fixed under the action of $\zed_v$. If we let $\zed_v$ act on any particular block in $\A$, we obtain an \emph{orbit} of blocks. It is well-known that a cyclic STS$(v)$ exists if and only if $v \equiv 1 \text{ or } 3 \bmod 6$, $v \neq 9$  (see \cite{CR}).
In the case of a cyclic STS$(v)$ with $v \equiv 1 \bmod 6$, there will be $(v-1)/6$ orbits, each consisting of $v$ blocks.
In the case of a cyclic STS$(v)$ with $v \equiv 3 \bmod 6$, there will be one \emph{short orbit} of $v/3$ blocks (consisting of 
the subgroup $\{0, v/3, 2v/3\}$ and its additive cosets) and the remaining orbits each consist of $v$ blocks.   

\subsection{Prior and related work}

Strong nestings have been considered implicitly in \cite{AGMMRRT,PBOM,BMNR} in the setting of harmonious colourings of Levi graphs of BIBDs. The paper by Buratti, Merola, Nakic and Rubio-Montiel \cite{BMNR} is the first one to note the connections between these two problems. We discuss these connections further in Section \ref{levi.sec}. 

In this paper, we are only considering the problem of nesting a $(v,k,\lambda)$-BIBD into a partial 
$(w,k+1,\lambda+1)$-BIBD. If we want to have $w = v$ (i.e., a minimal nesting), then 
$k \geq 2 \lambda + 1$ is a necessary condition. In other words, if $k < 2 \lambda+1$, then we must have $w > v$ 
(these are the cases we address in this  paper).

Marco Buratti pointed out that the standard embedding of an affine plane into a projective plane could be regarded as a type of nesting where the value of $\lambda$ does not change. 
Specifically, if we delete one line from a projective plane of order $n$, 
then we have a nesting of an $(n^2,n,1)-$BIBD into a partial $(n^2+n+1,n+1,1)$-BIBD. Both of these designs have $n^2 + n$ blocks.

Another approach is to require $w = v$, but permit $\lambda$ to increase by more than one. In other words, 
we could nest a $(v,k,\lambda)$-BIBD into a partial $(w,k+1,\lambda')$-BIBD, where $\lambda'$ is minimized.
For example, it is often possible to nest a $(v,3,\lambda)$-BIBD into a $(v,4,2\lambda)$-BIBD 
(for these parameters, the augmented design will be a BIBD). Many results on this problem can be found in \cite{CC83}. See also \cite{CR} for other results involving various types of nesting problems for triple systems.

\subsection{Our contributions}

We  introduce the notion of a weak nesting in this paper. In Section \ref{weak.sec}, we consider the simplest case, namely, weakly nesting $(v,2,1)$-BIBDs. We prove a lower bound on the number of points $w$ required in such a nesting and we provide constructions to show  that this bound can always be met with equality.

In Section \ref{weak>2}, we study weak nestings of $(v,k,\lambda)$-BIBDs for arbitrary $k \geq 2$. We first prove a lower bound on $w$. Then we study the case of $(v,3,2)$-BIBDs in detail.  It is well-known that a $(v,3,2)$-BIBD exists if and only if $v \equiv 0 \text{ or } 1 \bmod 3$.  For almost all values of $v \equiv 0,1 \text{ or } 3 \bmod 6$, we find $(v,3,2)$-BIBDs that have optimal weak nestings.  For $v \equiv 4  \bmod 6$, find $(v,3,2)$-BIBDs that have close to optimal weak nestings (with a small number of exceptions, $w$ is one greater than the lower bound).

In Section \ref{strong.sec}, we turn our attention to strong nestings.
The construction of optimal strong nestings of $(v,2,1)$-BIBDs for all $v$ was done by  Abreu {\it et al.} in \cite{AGMMRRT}. Here, we show how the lower bound on $w$ for strongly nested $(v,32,1)$-BIBDs from \cite{PBOM} can easily be proven by adapting our lower bound for weak nestings. We also show that some examples of optimal nestings can be obtained by modifying the weak nestings we constructed in Section \ref{weak.sec}.

In Section \ref{k>2.sec}, we investigate strongly nested $(v,k,\lambda)$-BIBDs for arbitrary $k \geq 2$. We give a simple alternate derivation of a lower bound on $w$ that was first proven in \cite{BMNR}. Then we focus on strong nestings of $(v,3,2)$-BIBDs.  For almost all values of $v \equiv  1, 3 \text{ or } 4 \bmod 6$, and for most values of $v \equiv 0 \bmod 12$, we find $(v,3,2)$-BIBDs that have optimal strong nestings. For $v \equiv  6  \bmod 12$, we  find $(v,3,2)$-BIBDs that have close to optimal strong nestings  (with a small number of exceptions, $w$ is two greater than the lower bound).

Section \ref{levi.sec} explores connections between strong nestings of BIBDs and harmonious colourings of Levi graphs of BIBDs, expanding on some results from \cite{BMNR}.  Finally, Section \ref{summary.sec} is a brief summary and conclusion.

\section{Optimal weak nestings of $(v,2,1)$-BIBDs}
\label{weak.sec}

The smallest case to consider is when $(X, \A)$ is a $(v,2,1)$-BIBD. A minimal nesting cannot exist from 
Lemma \ref{bound.lem}. We begin by considering weak nestings of $(v,2,1)$-BIBDs.
 In a $(v,2,1)$-BIBD, the set of blocks $\A$ just consists of all pairs of points from $X$. 
Suppose $\phi : \A \rightarrow Y$ defines a weak nesting, where $X \subseteq Y$ and $|Y| = w$. 
For convenience, for every $A \in \A$, denote $A' = A \cup \{\phi(A)\}$. Also, let $\A' = \{A' : A \in \A\}$.

We first establish a lower bound on $w$ as a function of $v$. 

\begin{Lemma}
\label{L1} Suppose a $(v,2,1)$-BIBD can be weakly nested into a partial $(w,3,2)$-BIBD. Then 
$w \geq  (5v-1)/4$.
\end{Lemma}

\begin{proof}
Let $(Y, \A')$ be the partial $(w,3,2)$-BIBD (i.e., the augmented design) obtained from weakly nesting the $(v,2,1)$-BIBD $(X, \A)$ using the mapping $\phi$.
First, observe that no block in $\A'$ contains two points from $Y \setminus X$. Hence, every point in $Y \setminus X$ occurs in at most $v$ blocks (since each pair of the form $\{x,y\}$, where $x \in X$ and $y \in Y \setminus X$,  can occur in at most two blocks in $\A'$).

We now obtain an upper bound on  $|\TTT|$, where $\TTT =  \{A \in \A : \phi(A) \in X \} $. Whenever $\phi(A) \in X$, $A'$ contains two new pairs of points from $X$. So \[|\TTT| \leq \frac{1}{2}\binom{v}{2}= \frac{v(v-1)}{4}.\]

The total number of blocks is at most 
\[ (w-v) v + |\TTT| \leq (w-v) v + \frac{v(v-1)}{4}.\] 
However, the total number of blocks is exactly $v(v-1)/2$, so we have
\begin{equation}
\label{eq1} (w-v) v + \frac{v(v-1)}{4} \geq \frac{v(v-1)}{2},\end{equation}
or
\[ (w-v)v \geq \frac{v(v-1)}{4},\]
which simplifies to give 
\[w \geq  \frac{5v-1}{4}.\]
\end{proof}

Since $w$ must be an integer, the minimum values of $w$ are as follows:
\begin{itemize}
\item if $v \equiv 1 \bmod 4$, say $v = 4t+1$, then $w \geq (5v-1)/4 = 5t+1$
\item if $v \equiv 3 \bmod 4$, say $v = 4t+3$, then $w \geq \lceil (5v-1)/4 \rceil = 5t+4$
\item if $v \equiv 0 \bmod 4$, say $v = 4t$, then $w \geq \lceil (5v-1)/4 \rceil = 5t$
\item if $v \equiv 2 \bmod 4$, say $v = 4t+2$, then $w \geq \lceil (5v-1)/4 \rceil = 5t+3$.
\end{itemize}

The case of equality in Lemma \ref{L1}, i.e., the case $v \equiv 1 \bmod 4$, is particularly interesting.
%---these will be optimal weak nestings. Certainly we require 
%These are the $v \equiv 1 \bmod 4$ in order for $(5v-1)/4$ to be an integer. 
In this situation, $w = (5v-1)/4$ and the augmented design is actually 
a $(w,3,2)$-BIBD with a hole of size $w-v$
(this just means that no block contains two points from $Y \setminus X$). 

Here are a few small examples of optimal weak nestings. Here and elsewhere, for brevity, we write the blocks of $\A'$ (the augmented design) in the form $(x,y,\phi(x,y))$.

\begin{Example}
\label{E1}
For $v = 5$, $w = 6$, develop the following two blocks modulo $5$: $(0,1 , \infty)$  and $(0,2,3)$.
Note that $\infty$ is the new point.
\end{Example}

\begin{Example}
\label{E2}
For $v = 9$, $w = 11$, develop the following four blocks modulo $9$: $(0,1 , \infty_1)$ $(0,2, \infty_2)$,
$(0,3,4)$ and $(0,4,6 )$. Note that $\infty_1, \infty_2$ are  new points.

\end{Example}

\begin{Example}
\label{E3}
For $v = 13$, $w = 16$, develop the following six blocks modulo $13$: $(0,1, \infty_1)$, $(0,5, \infty_2)$,   
$(0,6, \infty_3)$,  $(0,2,6)$,   $(0,3,5)$ and   $(0,4,1)$. Note that $\infty_1, \infty_2, \infty_3$ are  new points.
\end{Example}

We prove a general existence result for $v \equiv 1 \bmod 4$.

\begin{Lemma}
\label{T1mod4}
For any $v = 4t+1$, $v \geq 5$, there is an optimal weak nesting of the $(v,2,1)$-BIBD with $w = 5t+1$.
\end{Lemma}

\begin{proof}
Develop the following $2t$ blocks modulo $4t+1$ to obtain the augmented design:
\begin{center}
\begin{tabular}{ll}
$(0,2j, \infty_j)$, & for $j = 1, \dots , t$; and \\
 $(0,2j+1,2t-2j)$, & for $j = 0, \dots , t-1$,
 \end{tabular}
 \end{center}
where the $\infty_j$'s are the new points.
 
 The verifications are straightforward. 
 First, it is clear that the first two points in each block, when developed modulo $4t+1$, form a $K_{4t+1}$.
 Next, consider the differences of the form $\pm(2t-2j)$ and $\pm(2t-4j-1)$ (reduced modulo  $4t+1$) for 
 $j = 0, \dots , t-1$. The first set of differences is $\{ \pm 2, \pm 4, \dots , \pm 2t \}$ and  the second set of differences is $\{ \pm 1, \pm 3, \dots , \pm (2t-1)\}$. The development of all $2t$ blocks contains every pair in the $K_v$ once, and every pair of the form $\{i, \infty _j\}$ twice (and there are no pairs of the form $\{\infty_i, \infty _j\}$).
\end{proof}

\begin{Remark}
{\rm
A design is said to be \emph{t-pyramidal} when it has 
an automorphism group with $t$ fixed points, while acting 
sharply transitively on the others. 
Marco Buratti observed that the augmented design in Lemma \ref{T1mod4}
gives rise to a $t$-pyramidal $(5t+1,3,2)$-BIBD whenever
$t\equiv 0 \text{ or } 1 \bmod 3$: it suffices to add the blocks 
of a $(t,3,2)$-BIBD with point set $\{\infty_1,\dots,\infty_t\}$.
For instance, adding the block $\{\infty_1,\infty_2,\infty_3\}$  
twice to the augmented design of Example 2.3, we 
get a $3$-pyramidal $(16,3,2)$-BIBD. %In the next section,
%the augmented design obtained in Example 3.5 is a 
%full-fledged 3-pyramidal (10,3,2)-BIBD.
%Starting from [PR], there is a vast literature on 
%1-pyramidal (more commonly known as "1-rotational") 
%BIBDs but, as far as we are aware, 
Very little is 
known on $t$-pyramidal BIBDs with $t>1$ apart from 
two papers \cite{B2,B1} on $3$-pyramidal Steiner triple systems and Kirkman triple systems.
}
\end{Remark}

A slight modification of Lemma \ref{T1mod4} handles all the cases $v \equiv 3 \bmod 4$.

\begin{Lemma}
\label{T3mod4}
For any $v = 4t+3$, $v \geq 7$, there is an optimal weak nesting of the $(v,2,1)$-BIBD with $w = 5t+4$.
\end{Lemma}

\begin{proof}
Develop the following $2t+1$ blocks modulo $4t+3$  to obtain the augmented design: 
\begin{center}
\begin{tabular}{ll}
$(0,2j, \infty_j)$, & for $j = 1, \dots , t$;\\ $(0,1, \infty_{t+1})$; & and \\
 $(0,2j+1,2t+2-2j)$, & for $j = 1, \dots , t$,
 \end{tabular}
 \end{center}
 where the $\infty_j$'s are the new points.  The verifications are similar to those in the proof of Lemma \ref{T1mod4}. 

 \end{proof}
 
 For even values of $v$, we construct base blocks that are developed modulo $v-1$. The $v$ original points are
 $\zed_{v-1}$ and $\infty$, and the additional nested points are $\infty_i$, $1 \leq i \leq w-v$. 
 
 \begin{Lemma}
\label{T0mod4}
For any $v = 4t$, $v \geq 4$, there is an optimal weak nesting of the $(v,2,1)$-BIBD with $w = 5t$.
\end{Lemma}

\begin{proof}
Develop the following $2t$ blocks modulo $4t-1$ to obtain the augmented design: 
\begin{center}
\begin{tabular}{ll}
$(0,2j, \infty_j)$, & for $j = 1, \dots , t$;\\ 
$(\infty, 0, 2t-1)$; & and \\
 $(0,2j-1,t+j-1)$, & for $j = 1, \dots , t-1$,
 \end{tabular}
 \end{center}
 where the $\infty_j$'s  (but not $\infty$) are the new points.  The verifications are similar to those in the proof of Lemma \ref{T1mod4}. 

 \end{proof}
 
  \begin{Lemma}
\label{T2mod4}
For any $v = 4t+2$, $v \geq 6$, there is an optimal weak nesting of the $(v,2,1)$-BIBD with $w = 5t+3$.
\end{Lemma}

\begin{proof}
Develop the following $2t+1$ blocks modulo $4t+1$ to obtain the augmented design: 
\begin{center}
\begin{tabular}{ll}
$(0,2j, \infty_j)$, & for $j = 1, \dots , t+1$;\\ 
$(\infty, 0, 2t-1)$; & and \\
 $(0,2j-1,t+j-1)$, & for $j = 1, \dots , t-1$,
 \end{tabular}
 \end{center}
 where the $\infty_j$'s (but not $\infty$) are the new points.  The verifications are similar to those in the proof of Lemma \ref{T1mod4}. 

 \end{proof}

Combining Lemmas \ref{T1mod4}--\ref{T2mod4}, we have the following complete result.

\begin{Theorem}
For all $v$, there is an optimal weak nesting of the $(v,2,1)$-BIBD with $w = \left\lceil \frac{5v-1}{4} \right\rceil$.
\end{Theorem}

\section{Weak nestings of  BIBDs with block size $k\geq 3$}
\label{weak>2}

We prove a lower bound on $w$ for weak nestings of $(v,k,\lambda)$-BIBDs. The bound is useful in the cases where a minimal nesting does not exist, i.e., when $k \leq 2 \lambda$. In the following bound, recall that
$r = \lambda(v-1) / (k-1)$.

\begin{Theorem}
\label{genweak.thm}
Suppose %$r > (v-1)/2$ and suppose 
a $(v,k,\lambda)$-BIBD can be weakly nested into a partial $(w,k+1,\lambda +1)$-BIBD. Then 
\begin{equation}
\label{weak.ineq}
 w \geq \frac{r}{\lambda+1} + \frac{v + 2\lambda v + 1}{2(\lambda+1)} .
\end{equation}
\end{Theorem}

\begin{proof}
Let $(Y, \A')$ be the partial $(w,k+1,\lambda+1)$-BIBD obtained 
from a weak nesting of a $(v,k,\lambda)$-BIBD $(X, \A)$ using the mapping $\phi$.
We observe that the following properties hold:
\begin{enumerate}
\item Every point in $Y \setminus X$ occurs in at most 
$\left\lfloor \frac{(\lambda +1)v}{k} \right\rfloor$ blocks.  
\item 
Let $\TTT =  \{A \in \A : \phi(A) \in X \} $. Then $|\TTT| \leq \frac{1}{k}\binom{v}{2}.$
\item 
Therefore the  number of blocks is at most 
\[ (w-v) \left\lfloor \frac{(\lambda +1)v}{k} \right\rfloor + |\TTT| \leq (w-v) \left\lfloor \frac{(\lambda +1)v}{k} \right\rfloor + \frac{1}{k}\binom{v}{2}.\] 
\item Since $(X,\A)$ is a BIBD,  the %total 
exact number of blocks is  $\lambda v(v-1)/(k(k-1))$, so 
\begin{equation}
\label{weakeq2} (w-v) \left\lfloor \frac{(\lambda +1)v}{k} \right\rfloor  + \frac{1}{k}\binom{v}{2} \geq \frac{\lambda v(v-1)}{k(k-1)}.\end{equation}
Since $\left\lfloor \frac{(\lambda +1)v}{k} \right\rfloor \leq \frac{(\lambda +1)v}{k}$, equation (\ref{weakeq2}) implies
\begin{equation}
\label{weakeq3} (w-v)  \frac{(\lambda +1)v}{k}   + \frac{1}{k}\binom{v}{2} \geq \frac{\lambda v(v-1)}{k(k-1)}.
\end{equation}
\item Simplifying (\ref{weakeq3}), we have
\begin{align*} 
(w-v)(\lambda +1)    + \frac{v-1}{2} &\geq \frac{\lambda (v-1)}{k-1}\\
(w-v)(\lambda +1)    + \frac{v-1}{2} &\geq r\\
w-v &\geq \frac{1}{\lambda+1} \left( r - \frac{v-1}{2}\right)\\
w &\geq v + \frac{r}{\lambda+1} - \frac{v-1}{2(\lambda+1)}\\
w &\geq \frac{r}{\lambda+1} + \frac{v + 2\lambda v + 1}{2(\lambda+1)}.
\end{align*}
\end{enumerate}
\end{proof}

\begin{Remark}
{\rm When $k = 2$ and  $\lambda = 1$, we have $r = v-1$ and the inequality (\ref{weak.ineq}) becomes $w \geq (5v-1)/4$, agreeing with Lemma \ref{L1}.}
\end{Remark}

The case $k=3$, $\lambda = 1$ is the Steiner triple system case. Since $k \geq 2 \lambda +1$, Theorem \ref{genweak.thm} does not provide a nontrivial bound.  However, it was shown in \cite{St85} that there is a perfect  nesting of an STS$(v)$ (i.e., a nesting into a $(v,4,2)$-BIBD) for all $v \equiv 1 \pmod{6}$. %These are Banff designs with $k=3$, which have exact colourings, a fact that was also noted in \cite{BMNR}. 
On the other hand, an STS$(v)$ with $v \equiv 3 \pmod{6}$ cannot be nested (this follows from Lemma \ref{2k+1.lem}, but this fact has actually been known for a very long time). Lindner and Rodger \cite{LR} found perfect nestings of $3$-GDDs of type $3^{v/3}$ for all $v \equiv 3 \pmod{6}$, $v \geq 15$. By adding a new point to all the groups, we obtain an (optimal) strong nesting of an STS$(v)$ with $w = v+1$ for all $v \equiv 3 \pmod{6}$, $v \geq 15$. These nestings are also optimal when considered as weak nestings.

We next consider $\lambda = 2$. A $(v,3,2)$-BIBD cannot have a perfect nesting, so we consider weak nestings of $(v,3,2)$-BIBDs.  A $(v,3,2)$-BIBD exists if and only if $v \equiv 0 \text{ or } 1 \bmod 3$. Here $r = v-1$. Theorem \ref{genweak.thm} tells us that 
\begin{equation}
\label{3weakbound}
  w \geq \frac{v-1}{3} + \frac{5v + 1}{6}  = \frac{7v - 1}{6}. 
  \end{equation}
  Of course $w$ must be an integer, so we can strengthen (\ref{3weakbound}) to give
  \begin{equation}
\label{3weakbound-int}
  w \geq  \left\lceil \frac{7v - 1}{6} \right\rceil. 
  \end{equation}
  
\subsection{$v \equiv 1 \bmod 6$}
  
\begin{Theorem}
\label{1mod6}
Suppose $v \equiv 1 \bmod 6$. Then there exists a $(v,3,2)$-BIBD that has an optimal weak nesting.
\end{Theorem}

\begin{proof}
There is an STS$(v)$ that has a perfect nesting, say $(X, \A)$, from \cite{St85}.
Let $\phi : \A \rightarrow X$ be a nesting of this STS. 

Let $(X, \B)$ be a cyclic STS$(v)$. Since $v \equiv 1 \bmod 6$, a cyclic STS$(v)$, say $(X, \B)$, has $(v-1)/6$ orbits of $v$ blocks under the cyclic group $\zed_v$.
Denote the orbits by $\Or_1, \dots , \Or_{(v-1)/6}$. For $1 \leq i \leq (v-1)/6$, and for each block  $B \in \Or_i$, define $\phi(B) = \infty_i$ (the $\infty_i$'s are new points).

It is easy to see that $\phi$ defines a weak nesting of the $(v,3,2)$-BIBD $(X, \A \cup \B)$. Since we have introduced $(v-1)/6$ new points, we have $w = (7v-1)/6$ and the nesting is optimal, from (\ref{3weakbound}).
\end{proof}

\begin{Example}
\label{E7}
{\rm We construct an optimal weakly nested $(7,3,2)$-BIBD using Theorem \ref{1mod6}.
We start with a nested STS$(7)$: develop 
$(1,2,4,0)$ modulo $7$ (note that the last point in each augmented block is the nested point).
Now we take a cyclic STS$(7)$ and add a new point $\infty$ to each block, obtaining seven more blocks:
$(1,2,4,\infty)$ modulo $7$. The 14 blocks thus obtained form a weakly nested $(7,3,2)$-BIBD on $w = 8$ points.}
\end{Example}

\subsection{$v \equiv 3 \bmod 6$}

We first present an optimal nesting for $v = 9$.

\begin{Example}
\textup{(Marco Buratti \cite{Marco})}

\label{E9}
{\rm We construct an optimal weakly nested $(9,3,2)$-BIBD. We take two copies of the affine plane of order $3$ (i.e., AG$(2,3)$) on points $1, \dots , 9$.
Suppose one parallel class consists of the three blocks $\{1,2,3\}$, $\{4,5,6\}$ and $\{7,8,9\}$.
We nest the two copies of this parallel class as follows:
\[
\begin{array}{l@{\hspace{.5in}}l@{\hspace{.5in}}l}
(1,2,3,4) & (4,5,6,7) & (7,8,9,1) \\
(1,2,3,5) & (4,5,6,8) & (7,8,9,2) 
\end{array}
\]
For the remaining nine blocks in the first copy of AG$(2,3)$, adjoin $\infty_1$ to each block. 
For the remaining nine blocks in the second copy of AG$(2,3)$, adjoin $\infty_2$ to each block. 
Thus the augmented BIBD has $w = 11$ points; the new points are $\infty_1$  and $\infty_2$. The nesting is optimal, from (\ref{3weakbound-int}).
}
\end{Example}

%Now we consider $v \equiv 3 \bmod 6$.

\begin{Theorem}
\label{3mod6}
Suppose $v \equiv 3 \bmod 6$, $v \geq 9$. Then there exists a $(v,3,2)$-BIBD that has an (optimal) weak nesting with 
$w = (7v+3)/6$.
\end{Theorem}

\begin{proof}
Denote $v = 6t+3$. The case $t=1$ is done in Example \ref{E9}, so we can assume $t \geq 2$. 
For $t \geq 2$, there is a $3$-GDD of type $3^{2t+1}$ that has a perfect nesting, say $(X, \G, \A)$ (see \cite{LR}).
Let $\phi : \A \rightarrow X$ be a nesting of this GDD. %Note that $\G$ is a parallel class of points; we will treat the groups in $\G$ as blocks in the $(v,3,2)$-BIBD that we are constructing.

Let $(X, \B)$ be a  KTS$(v)$. $(X, \B)$ has $(v-1)/2 = 3t+1$ parallel classes.
Take one parallel class to be groups. We can assume that this parallel class is the same as $\G$. So we obtain a 
resolvable $3$-GDD of type $3^{2t+1}$, $(X, \G, \B)$.
There are $3t$ parallel classes in $\B$.

Let $\infty_i$, for $1 \leq i \leq t$, be new points.
Associate each $\infty_i$ for $1 \leq i \leq t$ with three parallel classes in $\B$, 
 in such a way that every parallel class has a unique $\infty_i$ associated with it.
For every block $A \in \B$, let $\phi(A)$ be the point associated with the parallel class containing $A$.

For each group $G \in \G$, take two copies of $G$ and nest both copies with $\infty$, where $\infty$ is another new point. It is easy to see that we have a weak nesting of the resulting $(v,3,2)$-BIBD. Since we have introduced $t+1$ new points, we have $w = 6t+3 + t+1 = 7t+4$ and the nesting is optimal, from (\ref{3weakbound-int}).
\end{proof}

\begin{Remark} 
{\rm It is also possible to prove Theorem \ref{3mod6} using the method of Theorem \ref{1mod6}.
Cyclic STS$(v)$ exist for $v \equiv 3 \bmod 6$, $v \geq 15$. 
Such a design will necessarily consist of one parallel class of blocks (a ``short orbit'' under the action of $\zed_v$) along with a collection of full orbits. The reader can fill in the remaining details.}
\end{Remark}

Our remaining constructions in this section, as well as the constructions for strong nestings of $(v,3,2)$-BIBDs, all follow a similar  strategy, which can be viewed as a generalization of the techniques used in proving 
Theorems \ref{1mod6} and \ref{3mod6}.
Suppose we are trying to find a (weak or strong) nesting of a $(v,3,2)$-BIBD.
We use the following approach.
\begin{enumerate}
\item
Find a $3$-GDD on $v$ points that can be nested.
\item Find a second $3$-GDD on $v$ points, having the same set of groups, in which the blocks can be partitioned into 
parallel classes (i.e., we have a resolvable GDD) or holey parallel classes (i.e., we have a frame).
Associate new points with the parallel or holey parallel classes of this GDD so the blocks of the GDD are nested (how this is done depends on whether we are seeking a weak or a strong nesting).
\item For each group $G$, find an appropriate nesting of a $(|G|,3,2)$-BIBD. In general, this requires additional new points. However, in the case where the GDD from step 2 is a frame,  it may possible to reuse some of the new points introduced in step 2.
\end{enumerate}

\subsection{$v \equiv 0 \bmod 6$}

We have an almost complete result for $v \equiv 0 \bmod 6$. We first construct a couple of small examples.

\begin{Example}
\label{E6}
{\rm We construct an optimal weakly nested $(6,3,2)$-BIBD.
Develop the following two base blocks modulo $5$ 
(the last point in each augmented block is the nested point):
$(0,1,3,\infty_1)$, $(\infty,0,1,2)$.
The $(6,3,2)$-BIBD has point set $\zed_{5} \cup \{\infty\}$. The augmented BIBD has $w = 7$ points; the  new point is $\infty_1$. The nesting is optimal, from (\ref{3weakbound-int}).
}
\end{Example}

\begin{Example}
\label{E12}
{\rm We construct an optimal weakly nested $(12,3,2)$-BIBD.
Develop the following four base blocks modulo $11$ 
(the last point in each augmented block is the nested point):
$(0,1,3,\infty_1)$, $(0,1,4,\infty_2)$, $(0,2,6,5)$, $(\infty,0,5,9)$.
The $(12,3,2)$-BIBD has point set $\zed_{11} \cup \{\infty\}$. The augmented BIBD has $w = 14$ points; the two new points are $\infty_1$ and $\infty_2$.  The nesting is optimal, from (\ref{3weakbound-int}).
}
\end{Example}

Our general construction for $v \equiv 0 \bmod 6$ makes use of nested and resolvable $3$-GDDs with all groups having size $6$. 

\begin{Theorem}
\label{0mod6}
There exists an optimal weakly nested $(6t,3,2)$-BIBD for all $t\geq 1$, $t \neq 3,6$. 
\end{Theorem}

\begin{proof}
The cases $t = 1,2$ were done in Examples \ref{E6} and \ref{E12}. So we can assume
$t \geq 4$, $t \neq 6$. For these values of $t$, it was shown in \cite[Lemma 3]{St85} that there exists a
nested $3$-GDD of type $6^t$, say $(X,\G,\A)$. Denote the nesting by $\phi : \A \rightarrow X$.

Next, take a resolvable $3$-GDD of type $6^t$, say $(X,\G,\B)$ (see \cite[Theorem 19.33]{CR}). We can assume that the points and groups are the same as in the nested GDD constructed above. This GDD has $3(t-1)$ parallel classes, which we denote
$\P_1, \dots , \P_{3(t-1)}$. Let $\infty_i$ ($1 \leq i \leq t-1$) be new points.  
We associate an $\infty_i$ with each parallel class in such a way that each $\infty_i$ is associated with three of the  parallel classes. We extend the nesting $\phi$ by defining  $\phi(B) = \infty_i$, where 
$\infty_i$ the new point associated with the parallel class containing $B$, for every block $B \in \B$.

Finally, replace each group $G \in \G$ by by a weakly nested  $(6,3,2)$-BIBD on the points in $G \cup \{\infty\}$, 
where $\infty$ is another new point
(see Example \ref{E6}).
The result is a weakly nested $(6t,3,2)$-BIBD with $w = 6t+ t - 1 + 1 = 7t$, which is optimal, from (\ref{3weakbound-int}).
\end{proof}

\subsection{$v \equiv 4 \bmod 6$}

We divide the case $v \equiv 4 \bmod 6$ into two subcases, $v \equiv 4 \bmod 12$ and $v \equiv 10 \bmod 12$.  In our constructions, we will be making use of $4$-GDDs with group sizes two and five. For up-to-date information on this topic, see \cite{ABBC}. We also use  $3$-frames of type $2^4$, which were constructed in \cite{St87}.

First we look at the case $v \equiv 4 \bmod 12$.
Our main theorem for $v \equiv 4 \bmod 12$ provides a construction in which $w$ is at most one greater than the optimal value.

\begin{Theorem}
\label{4mod12}
There exists a weakly nested $(12t+4,3,2)$-BIBD with $w = 14t+6$ for all $t\geq 2$. 
\end{Theorem}

\begin{proof}
It is known (e.g., see \cite{ABBC}) that there is a $4$-GDD of type $2^{3t+1}$ for any $t \geq 2$. Give every point weight two and apply the Fundamental Frame Construction from \cite{St87}, filling in $3$-frames of type $2^4$ for every block. The result is a $3$-frame of type $4^{3t+1}$.
This frame has $6t+2$ holey parallel classes.

Starting again from the $4$-GDD of type $2^{3t+1}$ and giving every point weight two, we apply Wilson's Fundamental GDD construction \cite[\S IV.2.1, Theorem 2.5]{CD}. We can replace every block by a nested $3$-GDD of type $2^4$ (see \cite{St87}). We obtain a nested $3$-GDD of type $4^{3t+1}$. We can assume that the groups of the nested $3$-GDD are the same as the groups of the $3$-frame.

Let $\infty_1, \dots , \infty_{2t+1}$ be new points. Associate a new point with every holey parallel class of the $3$-frame in such a way that every new point is associated with at most three holey parallel classes.
Then nest every block $A$ with the new point associated with the holey parallel class containing $A$.
We replace every group $G$ with a weakly nested $(4,3,2)$-BIBD on $G \cup \{\infty\}$, where $\infty$ is another new point (see Example \ref{E4}).
The result is a weakly nested $(12t+4,3,2)$-BIBD with $w = 12t+4 + 2t+1 + 1 = 14t+6$.

The lower bound on $w$ obtained in (\ref{3weakbound-int}) is \[w \geq \left\lceil \frac{7(12t+4) - 1}{6}\right\rceil  = \left\lceil \frac{84t+27}{6}\right\rceil = 14t+5.\]
Hence the value of $w$ in the nesting we have constructed is at most one greater than the optimal value.
\end{proof}

A variation of Theorem \ref{4mod12} can be used when $v \equiv 10 \bmod 12$. Here we again obtain nestings in which  $w$ is at most one greater than the optimal value. First we look at the case $v = 10$. 
\begin{Example}
\textup{(Marco Buratti \cite{Marco})}
\label{E10}
{\rm A weakly nested $(10,3,2)$-BIBD with $w = 12$. The BIBD is defined on points
$\mathbb{Z}_7 \cup \{ \infty_1,\infty_2,\infty_3\}$. The two new points are $\{A,B\}$.
The blocks in the augmented design are as follows:
\[
\begin{array}{l@{\hspace{.5in}}l@{\hspace{.5in}}l}
(1,2,4,0) \bmod 7& (\infty_1, \infty_2, \infty_3, 0) & (\infty_1, \infty_2, \infty_3, 6)\\ \hline
(\infty_1, 0,1, B) & (\infty_2, 0, 2, B) & (\infty_3, 0, 3, A) \\
(\infty_1, 1,2, \infty_2) & (\infty_2, 1, 3, A) & (\infty_3, 1, 4, \infty_1) \\
(\infty_1, 2,3, A) & (\infty_2, 2,4, \infty_3) & (\infty_3, 2,5, B) \\
(\infty_1, 3,4, B) & (\infty_2, 3,5, B) & (\infty_3, 3,6, B) \\
(\infty_1, 4,5, A) & (\infty_2, 4,6, A) & (\infty_3, 4,0, B) \\
(\infty_1, 5,6, B) & (\infty_2, 5,0, A) & (\infty_3, 5,1, A) \\
(\infty_1, 6,0, A) & (\infty_2, 6,1, B) & (\infty_3, 6,2, A) \\
\end{array}
\]
The result is an (optimal) weak nesting of a $(10,3,2)$-BIBD with $w =12$.
}
\end{Example}

\begin{Theorem}
\label{10mod12}
There exists a weakly nested $(12t+10,3,2)$-BIBD with $w = 14t+13$ for all $t\geq 2$. 
\end{Theorem}

\begin{proof}
It is known (see \cite{ABBC}) that there is a $4$-GDD of type $2^{3t}5^1$ for any $t \geq 3$. Give every point weight two and apply the Fundamental Frame Construction, filling in $3$-frames of type $2^4$ for every block. The result is a $3$-frame of type $4^{3t}10^1$.
This frame has $6t+5$ holey parallel classes.

Starting again from the $4$-GDD of type $2^{3t}5^1$ and giving every point weight two, we can replace every block by a nested $3$-GDD of type $2^4$. We obtain a nested $3$-GDD of type $4^{3t}10^1$. We can assume that the groups of the nested $3$-GDD are the same as the groups of the $3$-frame.  Let's name the groups $G_1, \dots , G_{3t+1}$, where $G_{3t+1}$ is the group of size $10$. There are two holey parallel classes associated with each of the first $3t$ groups and five holey parallel classes associated with the last group, $G_{3t+1}$.

Let $\infty_1, \dots , \infty_{2t}$ be new points. Associate a new point with the holey parallel classes of the $3$-frame corresponding to the first $3t$ groups in such a way that every new point is associated with exactly three holey parallel classes (note that $3t \times 2 = 2t \times 3$).
Then nest every block $A$ with the new point associated with the holey parallel class containing $A$.
For the last group, we use two additional new points, say $\infty_{2t+1}$ and $\infty_{2t+2}$. We associate $\infty_{2t+1}$ with three holey parallel classes and we associate $\infty_{2t+2}$ with two holey parallel classes.

We still need to replace the groups by nested $(4,3,2)$-BIBDs and a nested $(10,3,2)$-BIBD. For each group $G_i$ of size four, we need one new point for the nesting since the nested $(4,3,2)$-BIBD from Example \ref{E4} has $w = 5$. We will use a new point $\infty$.
For the last group, $G_{3t+1}$, of size ten, we need two new points to nest it since the nested $(10,3,2)$-BIBD from Example \ref{E10} has $w = 12$. However, here we can reuse $\infty_{2t+1}$ and $\infty_{2t+2}$ since neither of these points have occurred with any of the points in  $G_{3t+1}$.\footnote{Note that we cannot reuse any of $\infty_1, \dots , \infty_t$ since each of these points is associated with more than one hole.} The result is a weakly nested $(12t+4,3,2)$-BIBD with $w = 12t+10 + 2t+2 + 1 = 14t+13$.

The lower bound on $w$ obtained in (\ref{3weakbound-int}) is \[w \geq \left\lceil \frac{7(12t+10) - 1}{6} \right\rceil = \left\lceil\frac{84t+69}{6}\right\rceil = 14t + 12.\]
Hence the value of $w$ in the nesting we have constructed is at most one greater than the optimal value.
\end{proof}

\section{Strong nestings of BIBDs}
\label{strong.sec}

\subsection{Block size $k=2$}

We begin by briefly discussing strong nestings of $(v,2,1)$-BIBDs.
The next lemma is essentially the same as the lower bound from \cite[Theorem 11]{PBOM}. 
We give a simple proof using the technique we used to prove Lemma \ref{L1}.

\begin{Lemma}
\label{strongk=2}
\textup{\cite[Theorem 11]{PBOM}}
\label{L2}
Suppose a $(v,2,1)$-BIBD can be strongly nested. % into a partial $(w,3,2)$-BIBD. 
Then 
$w \geq  \lceil 3v/2 \rceil$.
\end{Lemma}

\begin{proof}
The only change from the proof of Lemma \ref{L1} is that now every point in $Y \setminus X$ occurs in at most 
$\lfloor \frac{v}{2} \rfloor$ blocks (since each pair of the form $\{x,y\}$, where $x \in X$ and $y \in Y \setminus X$,  can occur in at most \emph{one block} in $\A'$).  

Equation (\ref{eq1}) becomes 
\[(w-v) \left\lfloor \frac{v}{2} \right\rfloor + \frac{v(v-1)}{4} \geq \frac{v(v-1)}{2},\]
or 
\[(w-v) \left\lfloor \frac{v}{2} \right\rfloor  \geq \frac{v(v-1)}{4},\]
which simplifies to give
\[w  \geq v + \frac{v(v-1)}{4\left\lfloor \frac{v}{2} \right\rfloor}.\]
When $v$ is even, we obtain 
$w  \geq (3v-1)/2.$
Since $w$ is in integer, we must have
$w  \geq 3v/2 = \lceil 3v/2 \rceil.$
For $v$ odd, the resulting inequality is
$w  \geq 3v/2.$
Since $w$ is an integer, we have
$w  \geq (3v+1)/2 = \lceil 3v/2 \rceil.$
\end{proof}

The following complete result concerning (optimal) strong nestings of $(v,2,1)$-BIBDs
was proven in by Araujo-Pardo {\it et al.} in \cite[Theorem 1.2]{AGMMRRT}.

\begin{Theorem}
\textup{\cite[Theorem 1.2]{AGMMRRT}}
For any $v\geq 2$, $v \neq 4$, there is an (optimal) strong nesting of the $(v,2,1)$-BIBD with $w = \lceil {3v}/{2} \rceil$.
\end{Theorem}

%The proof given in \cite{AGMMRRT} is somewhat complicated. 
Here we just show how the optimal weak nestings of $(v,2,1)$-BIBDs constructed in Section \ref{weak.sec} can be modified to yield strong nestings of $(v,2,1)$-BIBDs. When $v \equiv 1 \bmod 4$,  the resulting strong nestings turn out to be optimal. 
First, we can easily modify the nestings in Examples \ref{E1},  \ref{E2} and  \ref{E3} to be strong nestings by increasing $w$ slightly. Here, the resulting values of $w$ are optimal (they meet the bound of Lemma \ref{L2}).

In Example \ref{E1}, the edges of the $K_{5}$ that are joined to $\infty$ form a cycle of odd length. This can be properly edge coloured using three colours. It follows that there is a strong nesting of the $(5,2,1)$-BIBD with $w = 5+3 = 8$. From Lemma \ref{L2}, this is optimal. 

In Example \ref{E2},   the edges of the $K_{9}$ that are joined to  $\infty_1$ or $\infty_2$ form a four-regular graph. By Vizing's Theorem, this subgraph can be properly edge coloured using at most five colours. It follows that there is a strong nesting of the $(9,2,1)$-BIBD with $w \leq 9 + 5 = 14$. From Lemma \ref{L2}, this is optimal.

In Example \ref{E3}, the edges of the $K_{13}$ that are joined to $\infty_1$, $\infty_2$ or $\infty_3$ form a six-regular graph. By Vizing's Theorem, this subgraph can be properly edge coloured using at most seven colours. It follows that there is a strong nesting of the $(13,2,1)$-BIBD with $w \leq 13 + 7 = 20$. From Lemma \ref{L2}, this is optimal. 

Using a similar technique, we have the following general result for $v \equiv 1 \bmod 4$.

\begin{Theorem}
For any $v = 4t+1$, there is an (optimal) strong nesting of the $(v,2,1)$-BIBD with $w = 6t+2$.
\end{Theorem}

\begin{proof}
Start with the same base blocks as in Theorem \ref{T1mod4}. The edges of the $K_v$ that are joined to an $\infty_j$ (for $1 \leq j \leq t$) form a $2t$-regular graph. By Vizing's Theorem, this subgraph can be properly edge coloured using at most $2t+1$ colours. It follows that there is a strong nesting of the $(v,2,1)$-BIBD with $w \leq 4t+1 + (2t+1) = 6t+2$. From Lemma \ref{L2}, this is optimal. 
\end{proof}

The above-described technique based on Vizing's Theorem can also be applied to the optimal weak nestings presented in Lemmas 
\ref{T3mod4}--\ref{T2mod4}, but the resulting strong nestings are not optimal.

\subsection{Block size $k>2$}
\label{k>2.sec}

\medskip
 We now turn our attention to strong nestings of  $(v,k,\lambda)$-BIBDs with $k > 2$.
It is straightforward to modify the proof of Theorem \ref{genweak.thm} to apply to the situation of strong nestings. In doing so, we obtain an alternate proof of the bound proven by Buratti, Merola, Nakic and Rubio-Montiel in 
\cite[Theorem 6.1]{BMNR}. %Note that every point in a $(v,k,\lambda)$-BIBD occurs in exactly $r = \lambda (v-1)/(k-1)$ blocks ($r$ is called the \emph{replication number} of the BIBD). 

\begin{Theorem}
\label{gen.thm}
\cite[Theorem 6.1]{BMNR}
Suppose %$r > (v-1)/2$ and suppose 
a $(v,k,\lambda)$-BIBD can be strongly nested into a partial $(w,k+1,\lambda +1)$-BIBD. Then 
\begin{equation}
\label{final.eq}
w \geq r + \frac{v+1}{2}.
\end{equation}
\end{Theorem}

\begin{proof}
As before, let $(Y, \A')$ be the partial $(w,k+1,\lambda+1)$-BIBD obtained from a strong nesting of a $(v,k,\lambda)$-BIBD $(X, \A)$ using the mapping $\phi$.
We observe that the following properties hold:
\begin{enumerate}
\item Every point in $Y \setminus X$ occurs in at most 
$\left\lfloor \frac{v}{k} \right\rfloor$ blocks (this is different from the corresponding statement in the proof of
Theorem \ref{genweak.thm}).
\item 
Let $\TTT =  \{A \in\A : \phi(A) \in X \} $. Then $|\TTT| \leq \frac{1}{k}\binom{v}{2}.$
\item 
Therefore the  number of blocks is at most 
\[ (w-v) \left\lfloor \frac{v}{k} \right\rfloor + |\TTT| \leq (w-v) \left\lfloor \frac{v}{k} \right\rfloor + \frac{1}{k}\binom{v}{2}.\] 
\item Since $(X,\A)$ is a BIBD,  the %total 
exact number of blocks is  $\lambda v(v-1)/(k(k-1))$, so 
\begin{equation}
\label{eq2} (w-v) \left\lfloor \frac{v}{k} \right\rfloor  + \frac{1}{k}\binom{v}{2} \geq \frac{\lambda v(v-1)}{k(k-1)}.\end{equation}
Since $\lfloor \frac{v}{k} \rfloor \leq \frac{v}{k}$, we have
\begin{equation}
\label{eq3} (w-v)  \frac{v}{k}   + \frac{1}{k}\binom{v}{2} \geq \frac{\lambda v(v-1)}{k(k-1)}.
\end{equation}
\item Simplifying (\ref{eq3}), we have
\begin{align*} (w-v)    + \frac{v-1}{2} &\geq \frac{\lambda (v-1)}{k-1}\\
w &\geq  v   + (v-1) \left( \frac{\lambda}{k-1} - \frac{1}{2} \right)\\
w & \geq r + \frac{v+1}{2},
\end{align*}
since $r = \lambda (v-1)/(k-1)$.
\end{enumerate}
\end{proof}

In the case $v = 2$, $\lambda = 1$, the bound of Theorem \ref{gen.thm} is the same as Theorem \ref{strongk=2}.
We should also note that the bound of Theorem \ref{gen.thm}  agrees with the bound for $k =2$ and general $\lambda$ that is proven in \cite[Theorem 1.4]{AGMMRRT}.

%\subsection{Block size $k=3$}
%\label{k=3strong.sec}

Now let's consider $(v,3,2)$-BIBDs. Here $r = v-1$, and so Theorem \ref{gen.thm} tells us that
\begin{equation}
\label{3strongbound}
w \geq v-1 + \frac{v+1}{2} = \frac{3v-1}{2}.
\end{equation}
Since $w$ is an integer, we can strengthen (\ref{3strongbound}) as follows:
\begin{equation}
\label{3strongbound-int}
w \geq \left\lceil \frac{3v-1}{2}\right\rceil.
\end{equation}

\subsection{$v \equiv 3 \bmod 6$}

We show how to get optimal (or close to optimal) strong nestings by modifying some of the constructions used to obtain weak nestings of $(v,3,2)$-BIBDs  in Section \ref{weak.sec}. We first look at the case $v \equiv 3 \bmod 6$.
We begin by showing that there are no strong nestings of $(v,3,2)$-BIBDs in which  $w = (3v-1)/2$.

\begin{Lemma}
\label{3mod6bound}
Suppose  $v\equiv 3 \bmod 6$. If a $(v,3,2)$-BIBD  has a strong nesting, then $w \geq (3v+1)/2$.
\end{Lemma}

\begin{proof}
The inequality (\ref{3strongbound}) states that $w \geq  (3v-1)/2$. 
If we have $w = (3v-1)/2$, then, in the proof of Theorem \ref{gen.thm}, all inequalities must in fact be equalities. This  will lead to a contradiction. 

Define $\TTT =  \{A \in \A: \phi(A) \in X \} $. Then $|\TTT| = v(v-1)/6$. 
Hence there are $v(v-1)/6$ blocks nested with old points and $v(v-1)/6$ blocks nested with new points.
Each of the $(v-1)/2$ new points must be attached to a parallel class of $v/3$ blocks. Hence, the blocks not in $\TTT$ contain every old point exactly $(v-1)/2$ times. Therefore, every old point occurs exactly $(v-1)/2$ times in a block in $\TTT$ in the $(v,3,2)$-BIBD.

Let's consider the multiset of pairs $P = \{ \{ x, \phi(A)\} : a \in A \in T\}$. This multiset must contain every pair of old points exactly once, so it is a set. Suppose an (old) point occurs $r$ times as a nested  point for a block in $\TTT$. The number of pairs in $P$ containing $x$ is $3r + (v-1)/2$. But there are $v-1$ pairs in $P$ that contain $x$, so 
$r = (v-1)/6$. However, this is not an integer, so we have a contradiction.
\end{proof}

\begin{Example}
\label{E9strong}
{\rm We construct a strongly nested $(9,3,2)$-BIBD with $w = 14$.
We take two copies of an AG$(2,3)$ defined on points $1, \dots , 9$. 
The five new points are $\infty_1, \dots , \infty_5$.

We have a total of eight parallel classes, each consisting of three disjoint blocks.
We can assume that three of the parallel classes are $123$, $456$, $789$; $159$, $267$, $348$; and $168$, $249$, $357$.
We nest these nine blocks as follows:
\[
\begin{array}{l@{\hspace{.5in}}l@{\hspace{.5in}}l}
(1,2,3,4)  & (4,5,6,7) & (7,8,9,1) \\ 
(1,5,9,2)  & (2,6,7,8) & (3,4,8,5) \\
(1,6,8,3)  & (2,4,9,6) & (3,5,7,9)
\end{array}
\]
There are five additional parallel classes, which we name $P_1, \dots , P_5$. We adjoin $\infty_i$ to all the blocks in $P_i$, for $1 \leq i \leq 5$. 
}
\end{Example}

\begin{Theorem}
\label{strong3}
Suppose $v \equiv 3 \bmod 6$, $v \geq 9$. Then there exists a $(v,3,2)$-BIBD that has an (optimal) strong nesting with $w = (3v+1)/2$. 
\end{Theorem}

\begin{proof}
Let $v = 6t+3$. The case $t=1$ is handled in Example \ref{E9strong}, so we can assume $t \geq 2$.

We adapt the construction given in the proof of Theorem \ref{3mod6}. Recall in the course of that proof that we constructed a resolvable $3$-GDD of type $3^{2t+1}$ that contained $3t$ parallel classes of blocks. We now associate a different new point with each of these parallel classes, and nest each block in one of these parallel classes with the relevant point. 
Then take two copies of each group, nest one with $\infty$ and nest the other with $\infty'$, where $\infty, \infty'$ are two additional new points. It is clear that we get a strong nesting of the $(v,3,2)$-BIBD with 
$w = 6t+3 + 3t + 2 = 9t+5$. This is optimal, in view of Lemma \ref{3mod6bound}.
\end{proof}

\subsection{$v \equiv 1 \bmod 6$}

In the case $v\equiv 1 \bmod 6$, we can construct optimal strong nestings for almost all values of $v$. 

\begin{Lemma}
\label{1mod6bound}
Suppose  $v\equiv 1 \bmod 6$. If a $(v,3,2)$-BIBD  has a strong nesting, then $w \geq (3v+1)/2$.
\end{Lemma}

\begin{proof}
The inequality (\ref{3strongbound-int}) states that $w \geq  (3v-1)/2$. 
However, we can prove a stronger result. 
When $v \equiv 1 \bmod 3$, we have
$\lfloor \frac{v}{3} \rfloor = (v-1)/3$ and the inequality (\ref{eq2}) becomes
\begin{equation*}
 (w-v) \left( \frac{v-1}{3} \right)  + \frac{1}{3}\binom{v}{2} \geq \frac{2 v(v-1)}{6}.\end{equation*}
This simplifies to give $w \geq 3v/2$. However, $3v/2$ is not an integer when $v \equiv 1 \bmod 6$, so it must be the case that $w \geq (3v+1)/2$. %Therefore the nesting we have constructed is optimal.
\end{proof}

\begin{Example}
\label{E7strong}
{\rm We construct an optimal  strongly nested $(7,3,2)$-BIBD with $w = 11$. This is optimal, in view of Lemma
\ref{1mod6bound}. 
The BIBD is defined on points
$\mathbb{Z}_6 \cup \{ \infty\}$. The four new points are $\{A,B,C,D\}$.
The blocks in the augmented design are as follows:
\[
\begin{array}{l@{\hspace{.5in}}l@{\hspace{.5in}}l}
(\infty,0,3,1) \bmod 6 & (0,2,4,D)  & (1,3,5, D) \\ 
(0,1,2,A)  & (1,2,3,B) & (2,3,4,C) \\
(3,4,5,A)  & (4,5,0,B) & (5,0,1,C)
\end{array}
\]
}
\end{Example}

\begin{Theorem}
\label{strong1}
Suppose $v \equiv 1 \bmod 6$, $v \geq 19$. Then there exists a $(v,3,2)$-BIBD that has an (optimal) strong nesting with 
$w = (3v+1)/2$. 
\end{Theorem}

\begin{proof} We adapt the proof of Theorem \ref{1mod6}, replacing the cyclic STS$(v)$ by  a Hanani triple system of order $v$, which exists from \cite{HATS}. A  Hanani triple system of order $v$ is an STS$(v)$ in which the blocks can be partitioned into $(v-1)/2$ near parallel classes (each consisting of $(v-1)/6$ disjoint blocks) and one partial parallel class consisting of $(v-1)/6$ blocks. We associate a new point with each of these $(v+1)/2$ partial parallel classes and proceed as before. This yields $(v,3,2)$-BIBD that has a strong nesting with 
$w = (3v+1)/2$. This is optimal from Lemma \ref{1mod6bound}.
\end{proof}

\subsection{$v \equiv 0 \bmod 6$}

We now turn to the case $v \equiv 0 \bmod 6$. Our first general construction yields strongly nested $(v,3,2)$-BIBDs in which $w$ is two greater than the lower bound. However, we will show a bit later that, for most values of $v \equiv 0 \bmod 12$, we can construct $(v,3,2)$-BIBDs with optimal strong nestings.

We begin by looking at the cases $v=6$ and $v=12$.

\begin{Example}
\label{strongE6}
{\rm We construct an optimal strongly nested $(6,3,2)$-BIBD with $w = 11$.
The points in the BIBD are $X = \zed_5 \cup \{\infty\}$. We introduce five new points, 
namely $\infty_i$, $i = 0,1,\dots , 4$.
To obtain the set of ten blocks, say $\A$, develop the following two base blocks modulo $5$ 
(the last point in each augmented block is the nested point):
$(0,1,3,\infty_0)$ (where the subscripts are also developed modulo $5$) and $(\infty,0,1,2)$.
The augmented BIBD has $w = 11$ points. 

The lower bound given by 
(\ref{3strongbound-int}) is $w \geq \lceil 17/2\rceil = 9$. 
Nevertheless, we can prove that $w = 11$ is optimal.  It is known that there is a unique $(6,3,2)$-BIBD up to isomorphism, and it is easily verified that any two blocks in this BIBD intersect. 

Let $\phi : \A \rightarrow Y$ be a nesting of this BIBD.
From the proof of Theorem \ref{gen.thm}, we have
 \[ | \{A : \phi(A) \in X \} | \leq \frac{1}{3}\binom{6}{2} = 5.\]
Any new point can be the nested point of at most one block, because there do not exist any pairs of disjoint blocks in the BIBD. Hence $|Y \setminus X| \geq 10 - 5 = 5$.  Therefore $w \geq 6 + 5 = 11$.
}
\end{Example}

\begin{Example}
\label{E12strong}
{\rm We construct a strongly nested $(12,3,2)$-BIBD with $w = 18$. The $(12,3,2)$-BIBD is constructed on points $1, \dots ,12$;  the new points are $13, \dots, 18$. Each new point is attached to a parallel class of blocks. The nesting of the remaining 20 blocks was found by computer. 
\[
\begin{array}{l@{\hspace{.5in}}l@{\hspace{.5in}}l@{\hspace{.5in}}l}
(1, 2, 3, 4) & (1, 2, 4, 5) & (1, 3, 4, 8) & ( 2, 3, 4, 6)\\
(5, 6, 7, 11) & (5, 6, 8, 10) & ( 5, 7, 8, 6) & ( 6, 7, 8, 12)\\
( 9, 10, 11, 3) & (9, 10, 12, 4) & (9, 11, 12, 10) & (10, 11, 12, 2) \\
( 1, 5, 9, 12) & (2, 6, 10, 1) & (3, 7, 11, 1) & ( 4, 8, 12, 11) \\
( 1, 6, 11, 9) & (2, 7, 12, 3) & ( 3, 8, 9, 5) & ( 4, 5, 10, 7) \\
( 1, 6, 12, 13) & (2, 5, 11, 13) & ( 3, 8, 10, 13) & ( 4, 7, 9, 13) \\
(1, 7, 10, 14) & (2, 8, 9, 14) & (3, 5, 12, 14) & ( 4, 6, 11, 14)\\
( 1, 8, 11, 15) & ( 2, 7, 12, 15) & (3, 6, 9, 15) & (  4, 5, 10, 15) \\
( 1, 7, 10, 16) & (2, 6, 9, 16) & (3, 5, 12, 16) & ( 4, 8, 11, 16) \\
(1, 8, 12, 17) & (2, 5, 11, 17) & (3, 6, 10, 17) & (4, 7, 9, 17) \\
( 1, 5, 9, 18) & (2, 8, 10, 18) & (3, 7, 11, 18) & ( 4, 6, 12, 18)
\end{array}
\]
}
\end{Example}

We have the following general result for $v \equiv 0 \bmod 6$. 

\begin{Theorem}
\label{strong0}
Suppose $v = 6t$ where $t \geq 4$, $t \neq 6$. Then there exists a $(v,3,2)$-BIBD that has a strong nesting with 
$w = 9t+2$. 
\end{Theorem}

\begin{proof} We adapt the proof of Theorem \ref{0mod6}. The resolvable $3$-GDD of type $6^t$ has $3(t-1)$ parallel classes. We associate a different new point with each of these parallel classes. This requires  $3(t-1)$ new points. 

Then replace each group $G \in \G$ by by a strongly nested  $(6,3,2)$-BIBD on the points in $G$ together with 
five new points,  say $\infty_i$, for $0 \leq i \leq 4$.
The result is a strongly nested $(6t,3,2)$-BIBD on $w = 6t + 3(t-1) + 5 = 9t + 2$ points.

The lower bound on $w$ from (\ref{3strongbound-int}) is 
\[ w \geq \left\lceil \frac{3(6t)-1}{2} \right\rceil = 9t.\]  
So our construction yields a value of $w$ that is two greater than the lower bound.
\end{proof}

We can prove a better result when $v \equiv 0 \bmod 12$. 

\begin{Theorem}
\label{strong0-12}
Suppose $v = 12t$ and $t \geq 4$. Then there exists a $(v,3,2)$-BIBD that has an (optimal) strong nesting with 
$w = 18t$. 
\end{Theorem}

\begin{proof} 
For $t \geq 4$, there exists a $4$-GDD of type $6^t$ (see \cite{ABBC}).
It was shown in \cite{St85} that there exists a
nested $3$-GDD of type $2^4$. Giving every point of the $4$-GDD weight two and applying Wilson's Fundamental GDD construction \cite[\S IV.2.1, Theorem 2.5]{CD}, we obtain a nested $3$-GDD of type $12^t$. Let $\phi$ denote the nesting of this GDD.

Next, take a resolvable $3$-GDD of type $12^t$, say $(X,\G,\B)$ (see \cite[Theorem 19.33]{CR}). We can assume that the points and groups are the same as in the nested GDD constructed above. This GDD has $6(t-1)$ parallel classes, which we denote
$\P_1, \dots , \P_{6(t-1)}$. Let $\infty_i$ ($1 \leq i \leq 6(t-1)$) be new points.  
We associate each $\infty_i$ with one parallel class. We extend the nesting $\phi$ by defining  $\phi(B) = \infty_i$, where 
$\infty_i$ the new point associated with the parallel class containing $B$, for every block $B \in \B$.

Finally, replace each group $G \in \G$ by by a strongly nested  $(12,3,2)$-BIBD on the points in $G \cup \{\alpha_i : 1 \leq i \leq 6\}$, 
where $\alpha_1, \dots , \alpha_6$ are six additional new points
(see Example \ref{E12strong}).
The result is a strongly nested $(12t,3,2)$-BIBD with \[w = 12t+ 6(t - 1) + 6 = 18t,\] which is optimal, from (\ref{3strongbound-int}).
\end{proof}

\subsection{$v \equiv 4 \bmod 6$}

Now we look at the case $v \equiv 4 \bmod 6$.
First we show that the lower bound from (\ref{3strongbound-int}) cannot be met with equality.

\begin{Lemma}
\label{4mod6bound}
Suppose  $v\equiv 4 \bmod 6$. If a $(v,3,2)$-BIBD  has a strong nesting, then $w \geq 3v/2 + 1$.
\end{Lemma}

\begin{proof}
The inequality (\ref{3strongbound-int}) states that $w \geq  3v/2$. 
However, we can prove a stronger result. 
When $v \equiv 4 \bmod 6$, we have
$\lfloor \frac{v}{3} \rfloor = (v-1)/3$ and the inequality (\ref{eq2}) becomes
\begin{equation*}
 (w-v) \left( \frac{v-1}{3} \right)  + \frac{1}{3}\binom{v}{2} \geq \frac{2 v(v-1)}{6}.
 \end{equation*}
 This simplifies to give $w \geq 3v/2$. Now we examine the case of equality in detail.
 
Suppose we have a strong nesting $\phi$ of a $(v,3,2)$-BIBD $(X, \A)$ with $w = 3v/2$. 
We use the same notation as in the proof of Theorem \ref{gen.thm}.
Let $\TTT =  \{A \in \A : \phi(A) \in X \} $. We must have $|\TTT| = v(v-1)/6.$
It follows that, for every pair $x, y \in X$, there is a unique block $B \in \TTT$ such that $x \in B$ and $\phi(B) = y$, or
$y \in B$ and $\phi(B) = x$.

The $w - v = v/2$ new points  each must occur in exactly $(v-1)/3$ blocks.
The blocks in $\A$ that contain a fixed new point form a near parallel class of $(X, \A)$. Thus the new points induce 
$v/2$ near parallel classes of blocks in $(X, \A)$.

Suppose $x \in X$ and suppose $x$ occurs $s_x$ times in blocks of $\TTT$ and $r_x$ times as a nested point of a block in $\TTT$.
Then 
\begin{equation}
\label{rands.eq}
3r_x + s_x = v-1.
\end{equation} 
There are $v$ points and $v(v-1)/6 $ blocks in $\TTT$. So there exists a point $x$ such that 
\[r_x \geq \left\lceil \frac{v(v-1)}{6v} \right\rceil = \frac{v+2}{6}.\]
For such a point $x$, (\ref{rands.eq}) yields the inequality
\[ s_x = v-1 - 3r_x \leq v-1 - \frac{v+2}{2} = \frac{v}{2} - 2.\]
Since  $x$ occurs in $v-1$ blocks in $\A$, it follows that $x$ occurs in at least
\[ v-1 - \left( \frac{v}{2} - 2 \right) = \frac{v}{2} + 1\] blocks in $\A \setminus \TTT$.
However, since $\A \setminus \TTT$ is partitioned into $v/2$ near parallel classes, it follows that $x$ can occur in at most
$v/2$ blocks in $\A \setminus \TTT$. Thus we have a contradiction, and hence $w \geq 3v/2 + 1$. 
\end{proof}

First we consider the subcase $v \equiv 4 \bmod 12$. We construct strong nestings where $w$ meets the lower bound from
Lemma \ref{4mod6bound}. So these are optimal nestings.

\begin{Theorem}
\label{strong4}
There exists an (optimal) strongly nested $(12t+4,3,2)$-BIBD with $w = 18t + 7$ for all $t\geq 2$. 
\end{Theorem}

\begin{proof}
We modify the proof of Theorem \ref{4mod12}.
As before, we construct a $3$-frame of type $4^{3t+1}$, which has
$6t+2$ holey parallel classes.
We also construct a nested $3$-GDD of type $4^{3t+1}$ where the 
groups of the nested $3$-GDD are the same as the groups of the $3$-frame.

We associate a different new point with every holey parallel class of the $3$-frame. This requires $6t+2$ new points.
To be precise, let's name the groups $G_1, \dots , G_{3t+1}$.
We can assume that we adjoin $\infty_{2i-1}$ and $\infty_{2i}$ to the holey parallel classes associated with $G_i$, for $1 \leq i \leq 3t+1$. 

Then we replace every group $G_i$ with a strongly nested $(4,3,2)$-BIBD (see Example \ref{E4strong}) on $G_i \cup \{\infty_{2i-1},\infty_{2i},\infty\}$, where $\infty$ is an additional new point .
The result is a strongly nested $(12t+4,3,2)$-BIBD with $w = 12t+4 + 6t+2 + 1 = 18t+7$.

The lower bound on $w$ obtained from Lemma \ref{4mod6bound} is $w \geq 3(12t+4)/2 + 1 = 18t+7$, so the nesting is  optimal.
\end{proof}

The subcase $v \equiv 10 \bmod 12$ is similar; we again construct optimal strong nestings. First we look at the case $v = 10$. %We present a nesting due to Marco Buratti \cite{Marco}.

\begin{Example} \textup{(Marco Buratti \cite{Marco})}
\label{E10strong}
{\rm 
We construct an optimal strongly nested $(10,3,2)$-BIBD with $w = 16$. The BIBD is defined on points
$\mathbb{Z}_7 \cup \{ \infty_1,\infty_2,\infty_3\}$. The six new points are $\{A,B,C,D,E,F\}$.

\[
\begin{array}{l@{\hspace{.5in}}l@{\hspace{.5in}}l}
(1,2,4,0) \bmod 7& (\infty_1, \infty_2, \infty_3, 1) & (\infty_1, \infty_2, \infty_3, 4)\\ \hline
(\infty_1, 0, 1, A) & (\infty_2, 0, 2, F) & (\infty_3, 0, 3, \infty_1) \\
(\infty_1, 1,2, B) & (\infty_2, 1, 3, E) & (\infty_3, 1, 4, F) \\
(\infty_1, 2,3, C) & (\infty_2, 2,4, A) & (\infty_3, 2,5, D) \\
(\infty_1, 3,4, D) & (\infty_2, 3,5, B) & (\infty_3, 3,6, A) \\
(\infty_1, 4,5, E) & (\infty_2, 4,6, C) & (\infty_3, 4,0, B) \\
(\infty_1, 5,6, F) & (\infty_2, 5,0, \infty_3) & (\infty_3, 5,1, C) \\
(\infty_1, 6,0, \infty_2) & (\infty_2, 6,1, D) & (\infty_3, 6,2, E) \\
\end{array}
\]
The nesting is optimal from Lemma \ref{4mod6bound}.
}
\end{Example}

\begin{Theorem}
\label{strong10}
There exists an (optimal) strongly nested $(12t+10,3,2)$-BIBD with $w = 18t+16$ for all $t\geq 0$, $t \neq 1$. 
\end{Theorem}

\begin{proof}
The case $t=0$ is done In Example \ref{E10strong}, so we only need to consider $t \geq 2$.
We adapt the proof of Theorem \ref{10mod12}.
As before, we construct a $3$-frame of type $4^{3t}10^1$, which
 has $6t+5$ holey parallel classes. We also construct a nested $3$-GDD of type $4^{3t}10^1$, where  
the groups of the nested $3$-GDD are the same as the groups of the $3$-frame.

Let $\infty_1, \dots , \infty_{6t+5}$ be new points. Associate each new point with a holey parallel class of the $3$-frame.
Then we nest every block $A$ with the new point associated with the holey parallel class containing $A$.

To be precise, let's name the groups $G_1, \dots , G_{3t+1}$, where $G_{3t+1}$ is the group of size $10$.
We can assume that we adjoin $\infty_{2i-1}$ and $\infty_{2i}$ to the holey parallel classes associated with $G_i$, for $1 \leq i \leq 3t$. For the last group, $G_{3t+1}$, there are five holey parallel classes; here we use the last five infinite points, 
$\infty_{6t+1}, \dots , \infty_{6t+5}$.

We still need to replace the groups by nested $(4,3,2)$-BIBDs and a nested $(10,3,2)$-BIBD. Here, unlike the weak nestings, we can reuse \emph{all} of the $\infty_i$'s. For each group $G_i$ of size four, we need three new points for the nesting since the nested $(4,3,2)$-BIBD from Example \ref{strong4} has $w = 7$. We can use $\infty_{2i-1}, \infty_{2i}$ and a new point $\infty$.
For the last group, of size ten, we need six new points for the nesting since the nested $(10,3,2)$-BIBD from Example \ref{strong10} has $w = 16$. We can use $\infty_{6t+1}, \dots , \infty_{6t+5}$ along with $\infty$. The result is a weakly nested $(12t+4,3,2)$-BIBD with $w = 12t+10 + 6t+5 + 1 = 18t+16$.

The lower bound on $w$ obtained from Lemma \ref{4mod6bound}  is $w \geq 3(12t+10)/2 + 1 = 18t + 16$. 
Hence the  nesting we have constructed is optimal.
\end{proof}

\section{Harmonious colourings of Levi graphs}
\label{levi.sec}

Strong nestings have received some previous attention in \cite{BMNR} in the context of harmonious colourings of Levi graphs. We briefly expand on some of these connections in this section.
 
 We begin with some relevant definitions that can be found in \cite{AGMMRRT,PBOM,BMNR}.
Let $(X, \A)$ be a $(v,k,\lambda)$-BIBD. The \emph{Levi graph} of $(X, \A)$, which we denote by $\LLL (X, \A)$, is the bipartite graph on vertex set
$X \cup \A$ having edges $\{x,A\}$ for all $x \in A \in \A$. A \emph{harmonious colouring} of  $\LLL (X, \A)$ is a proper colouring of the vertices of $\LLL (X, \A)$ such that every pair of colours appears on at most one edge. That is, we have a mapping $c : X \cup \A \rightarrow Y$ (for some set of colours $Y$) such that the following two properties are satisfied:
\begin{enumerate}
\item if $\{x,A\}$ is an edge of $\LLL (X, \A)$, then $c(x) \neq c(A)$, and
\item if $\{x,A\}$ and $\{x',A'\}$ are distinct edges in $\LLL (X, \A)$, then $\{ c(x),c(A) \} \neq \{ c(x'),c(A')\}$.
\end{enumerate}
If we assign a different colour to every vertex of $\LLL (X, \A)$, we of course obtain a harmonious colouring. In general, we want to find harmonious colourings in which the number of colours is minimized. This minimum is termed the \emph{harmonious chromatic number} of $\LLL (X, \A)$; it is denoted by $h(\LLL (X, \A))$.
 It is easy to see that $h(\LLL (X, \A)) \geq v$. 
  
 We first note a connection between strong nestings and harmonious colourings. The following result generalizes the discussion in \cite[\S 3]{BMNR} which points out the equivalence of perfect nestings of STS$(v)$ and harmonious colourings of the Levi graphs of STS$(v)$. We state our result in terms of strong nestings.

\begin{Theorem}
%\textup{\cite{BMNR}}
\label{BMNR.lem}
Suppose $(X, \A)$ is a $(v,k,\lambda)$-BIBD. Then there exists a {harmonious colouring} of  $\LLL (X, \A)$ using $w$ colours 
if and only if there is a strong nesting $\phi : \A \rightarrow Y$ where $X \subseteq Y$ and $|Y| = w$.
\end{Theorem}
\begin{proof}
Suppose $\phi : \A \rightarrow Y$ is a strong nesting of $(X, \A)$, where $|Y| = w$. We define a colouring of 
$\LLL (X, \A)$ using the colours in $Y$ as follows:
\begin{align*}
c(x) &= x \quad \text{for all $x \in X$}\\
c(A) &= \phi(A) \quad \text{for all $A \in \A$.}
\end{align*}
We prove that the colouring $c$ is harmonious. First,  if $\{x,A\}$ is an edge of $\LLL (X, \A)$, then
$c(x) = x$ and $c(A) = \phi(A) \neq x$ (this follows from property 1 of a strong nesting). Second, assume that 
$\{x,A\}$ and $\{x',A'\}$ are distinct edges in $\LLL (X, \A)$. Suppose that $\{ c(x),c(A) \} = \{ c(x'),c(A')\}$.
Then $\{ x, \phi(A) \} = \{ x', \phi(A')\}$, where $x \in A$ and $x' \in A'$. This contradicts property 2 of a strong nesting. 

Conversely, suppose that there exists a {harmonious colouring} of $c $ of $\LLL (X, \A)$ using $w$ colours.
First, we observe that $c(x) \neq c(x')$ if $x \neq x'$. To see this, suppose that $c(x) = c(x')$ for some $x \neq x'$. Let $A \in A'$ be a block that contains $x$ and $x'$. $\{x,A\}$ and $\{x',A\}$ are two edges in $\LLL (X, \A)$, and $\{ c(x),c(A) \} = \{ c(x'),c(A)\}$, which contradicts property 2 of a harmonious colouring.

Hence all $x\in X$ receive distinct colours. So we can assume without loss of generality that $c(x) = x$ for all $x \in X$. We can now define a strong nesting of $(X, \A)$ as follows:
$\phi(A) = c(A)$ for all $A \in \A$.

First, we show that $\phi(A) \cap A = \emptyset$ for all $A$. Suppose not; then $\phi(A) = x$ for some $x \in A$. But then property 1 of a harmonious colouring is contradicted.

Next, we show that $\{ \{ x, \phi(A)\}: x \in A  \} $ consists of {distinct pairs} of points. Suppose not; then
$\{ x, \phi(A)\} = \{ x', \phi(A')\}$ where $x \in A$ and $x' \in A'$. This means that  $\{ c(x),c(A) \} = \{ c(x'),c(A)\}$, which contradicts property 2 of a harmonious colouring.
\end{proof}

%We now recall several definitions from \cite{AGMMRRT,PBOM,BMNR}. 
%A \emph{harmonious colouring} of a graph $G$ is a proper vertex colouring of $G$ in which no two edges receive the same pair of colours. 
%The \emph{harmonious chromatic number} of a graph $G$, denoted by $h(G)$, is the minimum integer $k$ such that $G$ has a harmonious colouring using $k$ colours. 
A  \emph{Banff design}, as defined in \cite{BMNR},  is a 
 $(v,k,\lambda)$-BIBD, say $(X, \A)$, such that $h(\mathcal{L}(X, \A)) = v$. A Banff design has an exact colouring if every pair of colours occurs in exactly one edge of $\mathcal{L}(X, \A)$. This can occur only when $\binom{v}{2} = |X| \times |\A|$.
 
It was noted in \cite[\S 4]{BMNR} that Banff designs generated from so-called Banff difference sets have %Levi graphs with 
exact colourings. We prove a theorem that notes the equivalence of perfect nestings and Banff designs having exact colourings.

\begin{Theorem}
A $(v,k,\lambda)$-BIBD having a perfect nesting is equivalent to a $(v,k,\lambda)$-Banff design having an exact colouring.
\end{Theorem}

\begin{proof} 
Suppose a $(v,k,\lambda)$-BIBD has a perfect nesting. Then, from Theorem \ref{equiv.thm}, we have 
$\lambda = (k-1)/2$ and $r = (v-1)/2$. The Levi graph of the BIBD has $vr = \binom{v}{2}$ edges, so the nesting gives rise to an exact colouring.

Conversely, suppose a $(v,k,\lambda)$-Banff design has an exact colouring. 
As in the proof of Lemma \ref{BMNR.lem}, the vertices in the Levi graph receive distinct colours and
$\{ \{ x, \phi(A)\}: x \in A  \} $ consists of {distinct pairs} of points. Since there are $vr$ such pairs of points in the Levi graph,
the equation $\binom{v}{2} = vr$ yields $r = (v-1)/2$. Then $\lambda = r(k-1)/(v-1) = (k-1)/2$. Finally, since the colouring is exact, it immediately yields a perfect nesting of the $(v,k,\lambda)$-BIBD. 
\end{proof}

\begin{table}[tb]
\caption{Bounds on $w$ for weakly nested $(v,k,\lambda)$-BIBDs}
\label{tab-weak}
\[
\begin{array}{c|c|c|c|c|c|c}
k & \lambda & v & \text{lower bound} & \text{upper bound} & \text{exceptions} & \text{source}\\ \hline
2 & 1 & 4t & 5t & 5t & & \text{Lemma \ref{T0mod4}} \\
2 & 1 & 4t+1 & 5t+1 & 5t+1 & t \neq 0 & \text{Lemma \ref{T1mod4}}\\
2 & 1 & 4t+2 & 5t+3 & 5t+3 & t \neq 0 & \text{Lemma \ref{T2mod4}}\\
2 & 1 & 4t+3 & 5t+4 & 5t+4  & t \neq 0 & \text{Lemma \ref{T3mod4}}\\ \hline
3 & 2 & 6t+3 & 7t+4 & 7t+4  & t \neq 0,1 & \text{Theorem \ref{3mod6}}\\
3 & 2 & 6t+1 & 7t+1 & 7t+1  & t \neq 0 & \text{Theorem \ref{1mod6}}\\
3 & 2 & 6t & 7t & 7t  & t \neq 3,6& \text{Theorem \ref{0mod6}}\\
3 & 2 & 12t+4 & 14t+5 & 14t+6  & t \neq 0,1& \text{Theorem \ref{4mod12}}\\
3 & 2 & 12t+10 & 14t+12 & 14t+13  & t \neq 0,1& \text{Theorem \ref{10mod12}}\\ \hline
3 & 2 & 4 & 5 & 5  & & \text{Example \ref{E4}}\\
%3 & 2 & 9 & 11 & 11  & & \text{Example \ref{E9}}\\
3 & 2 & 10 & 12 & 12  & & \text{Example \ref{E10}}\\
\end{array}
\]
\end{table}

\begin{table}[tb]
\caption{Bounds on $w$ for strongly nested $(v,k,\lambda)$-BIBDs}
\label{tab-strong}
\[
\begin{array}{c|c|c|c|c|c|c}
k & \lambda & v & \text{lower bound} & \text{upper bound} & \text{exceptions} & \text{source}\\ \hline
3 & 2 & 6t+3 & 9t+5  & 9t+5 & t \neq 0& \text{Theorem \ref{strong3}}\\
3 & 2 & 6t+1 & 9t+2  & 9t+2 & t \neq 0,1,2& \text{Theorem \ref{strong1}}\\
3 & 2 & 6t & 9t  & 9t+2 & t \neq 1,2,3,6& \text{Theorem \ref{strong0}}\\
3 & 2 & 12t & 18t  & 18t & t \neq 1,2,3& \text{Theorem \ref{strong0-12}}\\
3 & 2 & 12t+4 & 18t+7  & 18t+7 & t \neq 0,1 & \text{Theorem \ref{strong4}}\\
3 & 2 & 12t+10 & 18t+16  & 18t+16 & t \neq 0,1 & \text{Theorem \ref{strong10}}\\ \hline
3 & 2 & 4 & 7  & 7 &  & \text{Example \ref{E4strong}}\\ 
3 & 2 & 6 & 11  & 11 &  & \text{Example \ref{strongE6}}\\ 
3 & 2 & 7 & 11  & 11 &  & \text{Example \ref{E7strong}}\\ 
3 & 2 & 10 & 16  & 16 &  & \text{Example \ref{E10strong}}\\ 
3 & 2 & 12 & 18  & 18 &  & \text{Example \ref{E12strong}}\\ 
\end{array}
\]
\end{table}

\section{Summary}
\label{summary.sec}

We have introduced the notion of weak nestings of BIBDs. 
A bound has been proven on the minimum value of $w$ such that a
$(v,k,\lambda)$-BIBD can be weakly nested into a partial $(w,k+1,\lambda+1)$-BIBD. We have shown that the bound can  be met with equality for $(v,2,1)$-BIBDs. For $(v,3,2)$-BIBDs with $v \equiv 0,1 \text{ or } 3 \bmod 6$, we have constructed optimal weak nestings (with a couple of exceptions). See Tables \ref{tab-weak} and \ref{tab-strong} for a summary of our constructions.

Strong nestings have previously been considered by various researchers. We have provided an alternative proof of a known lower bound on $w$ and we have shown how constructions for optimal weak nestings can sometimes be modified to yield optimal or near-optimal strong nestings. We have addressed the problem of finding strong nestings of $(v,3,2)$-BIBDs in detail.

Some problems that can be studied in future work include the following.

\begin{enumerate}
\item Construct optimal weak nestings of $(v,3,2)$-BIBDs with $v \equiv 4 \bmod 6$.
\item Construct optimal strong nestings of $(v,3,2)$-BIBDs with $v \equiv 6  \bmod 12$.
\item Construct optimal weak or strong nestings for BIBDs with block size $k \geq 4$ having $k \leq 2 \lambda$.
The first case to consider would be $(v,4,2)$-BIBDs.
\end{enumerate}

\section*{Acknowledgement}

Thanks to Marco Buratti for helpful comments and for providing Examples \ref{E9} and   \ref{E10}.
% and \ref{E10strong}.

\end{document}